\documentclass{amsart}

\usepackage{amsmath}
\usepackage{amssymb}
\usepackage{amsfonts}
\usepackage{amsopn}
\usepackage{amsthm}
\usepackage{amscd}
\usepackage[all]{xy}

\swapnumbers
\theoremstyle{plain}

\newtheorem{thm}[equation]{Theorem}
\newtheorem{lem}[equation]{Lemma}
\newtheorem{cor}[equation]{Corollary}
\newtheorem{prop}[equation]{Proposition}

\theoremstyle{definition}

\newtheorem{ex}[equation]{Example}
\newtheorem{rem}[equation]{Remark}
\newtheorem{dfn}[equation]{Definition}

\numberwithin{equation}{section}

\newcommand{\R}{\mathcal{R}}    % PSZ indecomposable motive
\newcommand{\LL}{\mathcal{L}}      % associated line bundle
\newcommand{\Gm}{\mathbb{G}_m}  % multiplicative group
\newcommand{\la}{\lambda}
\newcommand{\SL}{\mathrm{SL}}   % special linear group
\newcommand{\ach}{\alpha^\vee}
\newcommand{\bch}{\beta^\vee}
\newcommand{\dch}{\delta^\vee}
\newcommand{\ot}{\otimes}
\newcommand{\BB}{\mathfrak{B}} % variety of Borel subgroups

\newcommand{\qform}[1]{{\left\langle{#1}\right\rangle}}  % quadratic form
\renewcommand{\gg}[2]{\gamma^{{#1}/{#2}}}  % gamma quotient

\newcommand{\La}{\Lambda}
\newcommand{\ra}{\rightarrow}

\newcommand{\e}{\varepsilon}
\newcommand{\D}{\Delta}

\newcommand{\Gt}{\widetilde{G}}
\newcommand{\xit}{\tilde{\xi}}

\newcommand{\Z}{\mathbb{Z}}       % integers
\newcommand{\Zp}{\Z/p\Z}          % integers mod p
\newcommand{\Q}{\mathbb{Q}}      % rationals

\newcommand{\kalg}{k_{{\mathrm{alg}}}}

\DeclareMathOperator{\sa}{\mathbb{S}}   % symmetric algebra
\DeclareMathOperator{\CH}{CH}   % Chow group
\DeclareMathOperator{\Ch}{Ch}   % Chow group with Z/p coefficients
\DeclareMathOperator{\Chb}{\overline{\Ch}}  % the kernel of restriction
\DeclareMathOperator{\Tors}{Tors}  % torsion part
\DeclareMathOperator{\rank}{rank} % rank
\DeclareMathOperator{\im}{im}           % image

\DeclareMathOperator{\Br}{Br}
\DeclareMathOperator{\ind}{ind}

\DeclareMathOperator{\Sp}{Sp}

\DeclareMathOperator{\ed}{ed}    % essential dimension

\DeclareMathOperator{\car}{char} % characteristic

\DeclareMathOperator{\Aut}{Aut}

\DeclareMathOperator{\PSp}{PSp}
\DeclareMathOperator{\PGL}{PGL}

\newcommand{\cc}{\mathfrak{c}}     % characteristic map
\newcommand{\res}{\mathrm{res}}            % restriction map
\newcommand{\pr}{\mathrm{pr}}              % from Chow to tau
\newcommand{\aug}{\mathrm{aug}}

\newcommand{\fcase}[1]{{\underline{\emph{{#1}}:}}}

\begin{document}
\title{The $\gamma$-filtration
and the Rost invariant}

\author{Skip Garibaldi and Kirill Zainoulline}

\address{(Garibaldi) Department of Mathematics \& Computer Science, Emory University, Atlanta, GA 30322, USA}
\email{skip@member.ams.org}
\urladdr{http://www.mathcs.emory.edu/{\textasciitilde}skip/}

\address{(Zainoulline) Department of Mathematics and Statistics, University of Ottawa, 585~King Edward, Ottawa ON K1N6N5, Canada}
\email{kirill@uottawa.ca}
\urladdr{http://www.mathstat.uottawa.ca/{\textasciitilde}kirill/}

%%% Document STARTS here

\begin{abstract}
Let $X$ be the variety of Borel subgroups 
of a simple and strongly inner linear algebraic group $G$ over a field $k$.
We prove that the torsion part of the second quotient of 
Grothendieck's $\gamma$-filtration on $X$ is a cyclic group of order
the Dynkin index of $G$.

As a byproduct of the proof 
we obtain an explicit cycle $\theta$
that generates this cyclic group; we provide an upper
bound for the torsion of the Chow group of codimension-3 cycles on $X$; we relate the cycle $\theta$
with the Rost invariant and the torsion of the respective
generalized Rost motives; 
we use $\theta$ to obtain a uniform lower bound for the essential dimension of 
(almost) all simple linear algebraic groups.
\end{abstract}
\maketitle

Grothendieck's celebrated {\em $\gamma$-filtration} of the ring $K_0(X)$ 
gives a way to estimate the Chow group $\CH^*(X)$ of algebraic cycles on $X$ modulo
rational equivalence when $X$ is a smooth projective variety over a field $k$.  
Namely, by the Riemann-Roch theorem without denominators \cite[\S 15.3]{IT}
the $i$-th {\em Chern class} provides a well-defined group homomorphism
$$
c_i\colon \gg{i}{i+1}(X)\to \CH^i(X),\quad i\ge 0
$$ 
from the $i$-th quotient
of the $\gamma$-filtration to the Chow group 
of codimension-$i$ cycles on $X$.
Observe that for $i=0$ it is the identity map and for $i=1$
it is an isomorphism identifying $\CH^1(X)$ with the Picard group of $X$.

In the present paper we study these homomorphisms in the cases 
$i=2$, $3$ and
$X$ is a generically split
projective homogeneous variety under a semisimple
linear algebraic group $G$.
Our core determines and bounds respectively the torsion subgroup of $\gg23(\BB)$ and $\gg34(\BB)$ 
for the variety of Borel subgroups $\BB$ of strongly inner $G$ 
(Theorem~\ref{Borprop}).  
For instance, 
we show that  
the torsion subgroup $\Tors \gg23(\BB)$ 
is cyclic of order the Dynkin index
of $G$ and exhibit a 
generator $\theta$ for it (Definition~\ref{specc}). 

This fact together with the Riemann-Roch theorem imply
(see \S\ref{CH2.sec}) that 
the surjection
$$
\xymatrix{
\langle \theta \rangle=\Tors\gg23(\BB) \ar@{>>}[r]^-{c_2} & \Tors\CH^2(\BB)
}
$$ 
can be viewed 
as a substitute of the 
key map $Q(V)\to H^3(k,\Q/\Z(2))$ in the definition
of the Rost invariant \cite[pp.~126-127]{GMS}. 
Indeed, a theorem of Peyre-Merkurjev \cite{Peyre:deg3} shows that
$\Tors \CH^2(\BB)$ can be identified with the kernel of the restriction
$H^3(k,\Q/\Z(2))\to H^3(k(\BB),\Q/\Z(2))$.
Furthermore, the order of $c_2(\theta)$
in $\Tors\CH^2(\BB)$ 
equals to the order of the Rost invariant of $G$
(see Prop.~\ref{Gille.prop}). 

Our result gives bounds for the torsion in $\CH^3$ 
for generically split $X$ (see \S\ref{CH3.sec}) and
provides explicit generators of torsion subgroups of $\CH^2$ of certain
generalized Rost-Voevodsky motives. Note that typically, 
one does not even know a priori if the torsion subgroup of 
$\CH^i(X)$, $i\ge 3$, is finitely generated.  
However, determining the torsion subgroup determines 
$\CH^i(X)$ as an abelian group, since the dimension of its free part
$\CH^i(X) \ot \Q$ can be easily computed. 

 In section~\ref{sect:cohinv}, we study the behaviour of the image of $c_2$ under field extensions.
In particular, we show that this image is non-trivial if $G$ has Tits algebras of index 2 (Prop.~\ref{quaternionic}).
Using this fact
we prove (Prop.~\ref{proped}) that the essential dimension $\ed(G)$ of any absolutely almost simple linear algebraic group
not of type $A$ nor isomorphic to $\Sp_{2n}$ is greater or equal than $3$. 

%%%%%%%%%%%%%%%%%%%%%%%%%%%%%%%%%%%%%%%%%%%%%%%% 

\section{Preliminaries} \label{prelim.sec}

We now provide several facts and observations concerning
Chow groups, characteristic maps, invariants, Dynkin indices 
and filtrations on 
$K_0$ for
varieties of Borel subgroups of split simple linear algebraic groups.

\subsection{Two filtrations on $K_0$.}\label{twofil}

All facts provided here can be found in \cite[\S2]{Karp:cod}, 
\cite[\S15]{IT} and \cite[Ch.~3,5]{FuLa:ref}.
Let $X$ be a smooth projective variety
over a field $k$. Consider the $\gamma$-filtration on $K_0(X)$. 
It is given by the
subgroups
$$
\gamma^i(X)=
\langle c_{n_1}^{K_0}(b_1)\cdots c_{n_m}^{K_0}(b_m) \mid
\text{$n_1+\cdots + n_m\ge i$ and $b_1,\ldots,b_m\in K_0(X)$}\rangle,
$$
where $c_n^{K_0}$ denote the $n$-th Chern class with values in $K_0$.
For example, for the class of a line bundle 
we have $c_1^{K_0}([\LL])=1-[\LL^*]$.
Let $\gg{i}{i+1}(X)=\gamma^i(X)/\gamma^{i+1}(X)$ 
denote the respective quotient.
Consider the topological filtration on $K_0(X)$ given by the subgroups
$$
\tau^i(X)=\langle [\mathcal{O}_V]\mid 
\text{$V\hookrightarrow X$ and $\mathrm{codim}\; V\ge i$} \rangle,
$$
where $[\mathcal{O}_V]$ is the class of the structure sheaf of 
a closed subvariety $V$.
Let 
$\tau^{i/i+1}(X)=\tau^i(X)/\tau^{i+1}(X)$ 
denote the corresponding quotient.

There is an obvious surjection 
$\pr\colon \CH^i(X) \twoheadrightarrow \tau^{i/i+1}(X)$
from the Chow group of codimension $i$ cycles given by 
$V\mapsto [\mathcal{O}_V]$.
By the Riemann-Roch Theorem without denominators the $i$-th Chern class
induces the map in the opposite direction 
$$c_i\colon \tau^{i/i+1}(X)\to \CH^i(X)$$ and the composite
$c_i\circ \pr$ is the multiplication by $(-1)^{i-1}(i-1)!$ 
which is an isomorphism for $i\le 2$
\cite[Ex.15.3.6]{IT}.
For example, by the very definition we have 
$$
c_i(\prod_{j=1}^i c_1^{K_0}([\LL_j]))=
(-1)^{i-1}(i-1)! \prod_{j=1}^ic_1^{\CH}(\LL_j),
$$
where $\LL_j$ is a line bundle.
Observe also that $c_i$ becomes an isomorphism 
after tensoring with $\Q$.

There is an embedding
$\gamma^i(X)\subset \tau^i(X)$ for all $i$. 
Moreover, $\gamma^i(X)=\tau^i(X)$ for $i\le 2$.
Observe that $\gg12(X)=\tau^{1/2}(X)=\CH^1(X)$ is the Picard group and by \cite[Cor.~2.15]{Karp:cod}
there is an exact sequence
\begin{equation} \label{deg2.seq} 
0 \to \tau^3(X) / \gamma^3(X) \to 
\Tors \gg23(X) \xrightarrow{c_2} \Tors \CH^2(X) \to 0,
\end{equation}
where we have written $c_2$ for the composition
$\gamma^{2/3}(X)\to \tau^{2/3}(X) \xrightarrow{c_2} \CH^2(X)$.

\subsection{Characteristic maps and invariants}\label{q.def}

Let $G_s$ be a split simply connected 
simple linear algebraic group of rank $n$
over a field $k$. We fix a split maximal torus $T$
and a Borel subgroup $B$ such that $T\subset B\subset G_s$.
Let $\BB_s$ denote the variety of Borel subgroups of $G_s$ and
let $T^*$ denote the group of characters of $T$.
We fix a basis of $T^*$ 
given by the fundamental weights $\omega_1,\ldots,\omega_n$.

Let $\sa(T^*)$ be the symmetric algebra of $T^*$.
Its elements are polynomials in the fundamental weights $\omega_i$ 
with coefficients in $\Z$.
Let $\mathbb{Z}[T^*]$ be the integral group ring of $T^*$. 
Its elements are integral linear combinations 
$\sum_i a_ie^{\lambda_i}$, $\lambda_i\in T^*$. 
Consider the characteristic maps for $\CH$ and $K_0$ 
(see \cite[\S8, 9]{De73} and \cite[\S1.5, 1.6]{De74})
$$
\cc\colon \sa(T^*) \to \CH(\BB_s)
\text{ and }
\cc'\colon \Z[T^*] \to K_0(\BB_s)
$$
given by
$$
\cc\colon \omega_i \mapsto c_1^{\CH}(\LL(\omega_i)) 
\text{ and }
\cc'\colon e^\la \mapsto [\LL(\la)]\qquad{}
$$
where $\LL(\la)$ is the line bundle over 
$\BB_s$ associated to the character $\la$.

There are obvious augmentation maps $\sa(T^*)\to \Z$ and 
$\aug\colon \Z[T^*]\to \Z$ given
by $\omega_i\mapsto 0$ and $e^\la \mapsto 1$ respectively.
The Weyl group acts naturally on $T^*$, hence also on
$\sa(T^*)$ and $\Z[T^*]$. Consider the subrings of invariants 
$\sa(T^*)^W$ and $\Z[T^*]^W$. 
Let $I$ (resp.~$I'$) be the ideal 
generated by the elements of $\sa(T^*)^W$ 
(resp.~$\Z[T^*]^W$) from
the kernel of the augmentation map.
Then we have
$$
\ker \cc=I\text{ and }\ker\cc'=I'.
$$
Therefore we
have embeddings
$$
\cc\colon \sa(T^*)/I \hookrightarrow \CH(\BB_s)
\text{ and }
\cc'\colon \Z[T^*]/I' \xrightarrow{\simeq} K_0(\BB_s),
$$
where the second map is surjective since $G_s$ is simply connected 
(see \cite{Pi72}).

By \cite[\S2 and Cor.2]{De73} the kernel $I$ of $\cc$ consists
of elements $g$ such that 
\begin{equation} \label{g.eq}
m\cdot g=\sum_i g_i\cdot f_i,
\end{equation}
for $m \in \Z$, 
$f_i$ the basic polynomial invariants, and $g_i\in \sa(T^*)$.

There is a $W$-invariant quadratic form $q$ on $T^* \ot \Q$ 
that is uniquely determined up to a scalar multiple 
\cite[\S{VI.1.1--2}]{Bou:g4}. 
We normalize $q$ so that it takes the value 1 on every short coroot; 
as $q$ is indivisible, it can be taken as the generator of $I$ of degree 2.  
To say it differently, 
each element of $I$ of degree 2 is a multiple of $q$ 
by an integer.

The form $q$ should be familiar.  Its 
polar bilinear form $b_q$
amounts to the restriction of the ``reduced Killing form" 
to the Cartan subalgebra of the Lie algebra of $G_s$ 
as described in \cite[\S5]{GrossNebe}.
In the case where the roots are all one length, 
an explicit formula for $b_q$ is well known: its Gram 
matrix is the Cartan matrix of the root system.

\subsection{Degree 3 elements of $I$} \label{deg3}
If $G_s$ is not of type $A_n$ ($n \ge 2$), then there is no basic invariant of degree 3 \cite[p.~59]{Hum:ref}.  Then by \eqref{g.eq} and the indivisibility of $q$, every $g \in I$ of degree 3 can be written as $g = (\sum a_i \omega_i) q$ for some $a_i \in \Z$.

If $G_s$ is of type $A_n$ for some $n \ge 2$, then $\sa(T^*)^W$ has a basic generator $f_3$ of degree 3.  We view the weight lattice as in \cite{Bou:g4}, meaning that it is contained in a lattice with basis $\e_1, \ldots, \e_{n+1}$ so that the embedding is defined by $\e_1 = \omega_1$, $\e_{n+1} = -\omega_n$, and $\e_i = \omega_i - \omega_{i-1}$ for $2 \le i \le n$, and $W$ is the permutation group of the $\e$'s.  Then $q$ and 
\[
f_3 := (\e_1^3 + \cdots + \e_{n+1}^3)/3 = \sum_{i=2}^n \omega_{i-1}^2 \omega_i - \omega_{i-1} \omega_i^2
\]
are members of a set of basic invariants \cite[\S3.12]{Hum:ref}.

For $g \in I$ of degree 3, we have $mg = g_2 q + g_3 f_3$ for some $g_3, m \in \Z$ and $g_2 = \sum a_i \omega_i$ with $a_i \in \Z$.  On the right side, the monomial $\omega_i^3$ occurs only in $g_2 q$ and has coefficient $a_i$, hence $m$ divides $a_i$ for all $i$, hence $m$ divides $g_2$ and also $g_3$.  In summary, $g = (g_2/m) q + (g_3/m) f_3$ for some $g_2/m, g_3/m \in \sa(T^*)$.

\subsection{The $\gamma$-filtration on the variety of Borel subgroups.}\label{gamBorel}

Consider the $\gamma$-filtration 
on the variety $\BB_s$ of Borel subgroups  of $G_s$.
Let $\gamma^m$ denote the subgroup of $\Z[T^*]$ 
generated by products of at least
$m$ elements of the form $(1-e^{-\omega_i})$, 
where $\omega_i$ is a fundamental weight.
Then the isomorphism $\cc'$ induces an isomorphism 
$$
\gg{m}{m+1}(\BB_s)\simeq \gamma^m/(\gamma^{m+1}+I') 
\text{ for each }i.
$$
For example  $\gg12(X)\simeq \gamma^1/(\gamma^2+I')$ 
is a free abelian group
with a basis given
by the classes of the elements
$$
(1-e^{-\omega_i})\in \gamma^1, \quad i=1,\ldots, n.
$$
Indeed, $c_1^{K_0}([\LL(\omega_i)])=1-[\LL(-\omega_i)]$, 
the map $c_1\colon \gg12(\BB_s)\to \CH^1(\BB_s)$
is an isomorphism
and the elements $c_1(\LL(\omega_i))$ for $i=1,\ldots, n$ 
form a basis of the Picard group $\CH^1(\BB_s)$.

Since $K_0(\BB_s)$ is generated
by classes of line bundles (see \cite{Pi72}), 
so is $\gamma^i(\BB_s)$. Therefore, we have
$$
\gamma^i(\BB_s)=
\langle c_1^{K_0}([\LL_1])\cdots c_1^{K_0}([\LL_m])
\mid \text{$m\ge i$ and 
$\LL_j$ is a line bundle over $\BB_s$} \rangle.
$$
Let $\la=\sum_i a_i\omega_i$ be a presentation of 
a character $\la$
in terms of the fundamental weights. 
Then $\LL(\la)=\ot_i \LL(\omega_i)^{\ot a_i}$.
Since for any two line bundles $\LL_1$ and $\LL_2$ we have
$$
c_1^{K_0}([\LL_1\ot\LL_2])=c_1^{K_0}([\LL_1])+c_1^{K_0}([\LL_2])-
c_1^{K_0}([\LL_1])c_1^{K_0}([\LL_2])
$$
applying this formula recursively 
we can express any element of $\gg{i}{i+1}(\BB_s)$
as a linear combination of the products 
of the first Chern classes of the bundles 
$\LL(\omega_i)$, $i=1\ldots n$.
For instance, any element of $\gg23(\BB_s)$ 
can be written as a class of
$$
\sum_{i=1}^n \sum_{j=1}^n a_{ij}(1-e^{-\omega_i})(1-e^{-\omega_j}) \in 
\gamma^2\mod \gamma^3+I',
\text{ where } a_{ij}\in \Z.
$$

\subsection{The Dynkin index.}\label{index.def}

Let $N$ denote the map $\Z[T^*]^W \to \Z$ 
defined by fixing a long root $\alpha$ and setting 
\[
N\left( \sum\nolimits_i a_i e^{\la_i} \right) := 
\frac12 \sum\nolimits_i a_i \qform{\lambda_i, \ach}^2.
\]
This does not depend on the choice of $\alpha$ and 
takes values in $\Z$ (and not merely in $\frac12 \Z$), 
cf.~Lemma \ref{equinum} below.  
The number $N(\chi)$ is called the \emph{Dynkin index of $\chi$}.  
Note that for $m \in \Z$, we have $N(m) = N(me^0) = 0$, so 
$N(\chi)$ only depends on the image of $\chi$ in the kernel 
of the augmentation map.  

In case $G_s$ has two root lengths, 
it is natural to wonder what one would find 
if one used a short root, say, $\delta$ in the definition of $N$ 
instead of the long root $\alpha$.  
We claim that
\begin{equation}\label{longroot}
\frac12 \sum a_i \qform{\la_i, \dch}^2 = 
q(\dch) \left[ \frac12 \sum a_i \qform{\la_i, \ach}^2 \right],
\end{equation}
where $q$ is the form introduced in \ref{q.def}.
In other words, one obtains something 
that differs by a factor of $q(\dch)$.  
(We will use this observation later.)  
To prove it, define quadratic forms $n_\alpha$ and $n_\delta$ on $T^*$ 
via $n_\alpha(\la) = \sum_{w\in W} \qform{w\la, \ach}^2$ and 
similarly for $\dch$.  
For example, $n_\delta(\alpha) = q(\dch)^2 n_\alpha(\delta)$.  
But $n_\alpha$ is a $W$-invariant quadratic form on $T^*$, 
hence it is a scalar multiple of $q$.  As $q(\alpha) = q(\dch) q(\delta)$, we have 
 $n_\alpha(\alpha) = q(\dch) n_\alpha(\delta)$.  
But $n_\delta$ is also a scalar multiple of $q$, so
we conclude that $n_\delta = q(\dch) n_\alpha$, proving the claim.  

The \emph{Dynkin index} $N(G_s)$ 
is defined to be the gcd of $N(\chi)$ as $\chi$ 
varies over the characters of finite-dimensional representations of 
$G_s$.
The number $N(G_s)$ is calculated 
in \cite{GMS}, \cite{KumarNarasimhan}, or \cite{LaszloSorger}:
\[
\begin{tabular}{c||c|c|c|c|c} \hline
type of $G_s$& $A$ or $C$& $B_n$ ($n \ge 3$), 
$D_n$ ($n \ge 4$), $G_2$ & $F_4$ or $E_6$&$E_7$&$E_8$ \\ \hline
$N(G_s)$&1&2&6&12&60\\ \hline
\end{tabular}
\]

If $G$ is a simple and strongly inner group, 
then, for the purposes of this paper, we define the Dynkin index $N(G)$ of $G$ to be 
the Dynkin index $N(G_s)$ of the split simply connected group 
of the same Killing-Cartan type.  

%%%%%%%%%%%%%%%%%%%%%%%%%%%%%%%%%%%%%%%%%%%%%%%%%%%%%%%%%%%

\section{Dynkin indices and the map $\phi$} \label{index.sec}

Let $G_s$ denote a split simply connected simple linear
algebraic group of rank $n$ over a field $k$. 
We fix a pinning for $G_s$ and in particular a split maximal torus $T$ 
and fundamental weights $\omega_1, \ldots, \omega_n$.  
As $G_s$ is simply connected, 
$T_*$ ($= \mathrm{Hom}(\Gm, T)$) and $T^*$ are canonically identified 
with the coroot and weight lattices respectively.

\begin{dfn}
Put $\Z[T^*] := \Z[e^{\omega_1}, \ldots, e^{\omega_n}]$, 
the integral group ring, and $\sa(T^*) := \Z[\omega_1, \ldots, \omega_n]$, 
the symmetric algebra of $T^*$.  
We define a ring homomorphism 
\[
\phi_m \colon \Z[T^*]/\gamma^{m+1}\to \sa(T^*)/(\sa^{m+1}(T^*)),\; m\ge 2,
$$
$$
\text{via }
\phi_m\left(e^{\sum_{i=1}^n a_i \omega_i}\right) =\prod_{i=1}^n(1-\omega_i)^{-a_i}.
\]
In particular, $\phi_m(e^{\omega_i})=1+\omega_i+\cdots +\omega_i^m$ and
$\phi_m(e^{-\omega_i})=1-\omega_i$. 
(Note that $\Z[T^*]$ can be viewed as Laurent polynomials 
in the variables $\omega_1, \ldots, \omega_n$, 
and from this perspective it is clear that the formula for $\phi$ 
gives a well-defined ring homomorphism on $\Z[T^*]$ and 
$\phi_m(\gamma^{m+1})$ is zero in $\sa(T^*)/(\sa^{m+1}(T^*))$.)  
\end{dfn}

The homomorphism $\phi_m$ is an isomorphism.  
To see this, define a homomorphism $\sa(T^*) \to \Z[T^*]/\gamma^{m+1}$ 
via $\psi_m(\omega_i) = 1 - e^{-\omega_i}$ 
for all $i$; it induces a homomorphism $\sa(T^*)/(\sa^{m+1}(T^*)) 
\to \Z[T^*]/\gamma^{m+1}$ that we also denote by $\psi_m$.  
As the compositions $\phi_m \psi_m$ and 
$\psi_m \phi_m$ are both the identity on generators, the claim is proved.

The goal of this section is to give a formula for $\phi_2$ on $I'$. 

\begin{prop}\label{maincomp} 
If $G_s$ is simple, then
for $\chi \in \Z[T^*]^W$, we have:
\[
\phi_2(\chi) = 
\aug(\chi) + N(\chi) \cdot q \quad \in \left( \sa(T^*)/(\sa^3(T^*)) \right)^W,
\]
where $q$ is the invariant form introduced in \S\ref{q.def}.
\end{prop}

The proof would be much easier if we already knew that $\phi_2$ takes $W$-invariant elements to $W$-invariant elements, but this only comes as a consequence of the proof of the proposition.  We give some preliminary material before the proof.  

\begin{ex}[$\SL_2$]  \label{A1.eg}
In case $G_s = \SL_2$, 
write $\omega$ for the unique fundamental weight.  
For $n > 0$, we have:
\[
\phi_2(e^{n\omega} + e^{-n\omega}) = 
(1 + \omega + \omega^2)^n + (1-\omega)^n = 2+n^2 \omega^2,
\]
which verifies Prop.~\ref{maincomp} for this group.
\end{ex}

\begin{ex}[$\SL_2 \times \SL_2$]\label{A1A1.eg}
In case $G_s = \SL_2 \times \SL_2$ 
there are two fundamental weights $\omega_1, \omega_2$ and 
the Weyl group $W$ is the Klein four-group; 
it acts by flipping the signs of $\omega_1$ and $\omega_2$. 
The definition of $\phi_2$ above makes sense here even though $G_s$ is not simple.  
We find:
\[
\phi_2(We^{a_1 \omega_1 + a_2 \omega_2}) = 
4 + 2\left[ a_1^2 \omega_1^2 + a_2^2 \omega_2^2\right].
\]
\end{ex}

One final observation about Weyl group actions.  
We write $W\la$ for the $W$-orbit of $\la \in T^*$.

\begin{lem} \label{equinum}
For every root $\alpha$ and weight $\la \in T^*$, 
the map $W\la \to \Z$ defined by $\pi \mapsto \qform{\pi, \ach}$ 
hits $x$ and $-x$ the same number of times, for every $x \in \Z$.    
If $\alpha, \beta$ are orthogonal roots, then for every weight 
$\la \in T^*$, the map $W \la \to \Z \times \Z$ defined by 
$\pi \mapsto (\qform{\pi, \ach}, \qform{\pi, \bch})$ 
hits $(x,y)$, $(-x,y)$, $(x,-y)$, and $(-x,-y)$ 
the same number of times, for every $x,y \in \Z$.
\end{lem}

\begin{proof}[Sketch of proof]
It is an exercise to show the analogous statements for the map $W \ra \Z$ defined by $w \mapsto \qform{w\la, \ach}$ and similarly for the second claim.  The lemma follows.
\end{proof}

\begin{proof}[Proof of Prop.~\ref{maincomp}]
We may assume that $\chi = \sum e^{\la_j}$ 
where $\la_1, \ldots, \la_r$ is the Weyl orbit of some $\la \in T^*$.  
Put $\la_j = \sum_{i=1}^n a_{ij} \omega_i$, 
so $\phi(\chi) = \sum_{j=1}^r \prod_{i=1}^n (1 - \omega_i)^{-a_{ij}}$. 
Obviously, the degree $0$ component of $\phi(\chi)$ is $r = \aug(\chi)$.

The degree $1$ component of $\phi(\chi)$ is 
$\sum_j \sum_i a_{ij} \omega_i = \sum_i \left( \sum_j a_{ij} \right) \omega_i$.  
Here the claim is that $\sum_j a_{ij} = 0$ for each $i$.  
The $a_{ij}$'s are the images of $W\la$ in $\Z$ under the map 
$\la_j \mapsto \qform{\la_j, \alpha_i^\vee}$ 
where $\alpha_i$ denotes the simple root corresponding 
to the fundamental weight $\omega_i$, 
hence the claim follows from Lemma \ref{equinum}.

The crux is to check the claim on the degree $2$ component 
$q_1$ of $\phi(\chi)$; 
it is an integer-valued quadratic form 
on the coroot lattice $T_*$ and we check that it equals $q_2 := N(\chi)q$. 
We write out for $\ell = 1, 2$:
\begin{equation} \label{mc.q}
q_\ell(\sum d_i \ach_i) = 
\sum\nolimits_i d_i^2 q_\ell(\ach_i) + \sum\nolimits_{i<j} d_i d_j b_{q_\ell}(\ach_i, \ach_j),
\end{equation}
where $b_{q_\ell}$ is the polar bilinear form of $q_\ell$.  
We will check that the value of this expression is the same for $\ell = 1, 2$.

First suppose that $\dch := \sum d_i \ach_i$ is a coroot and every $d_i$ is 0 or 1.  
Then it defines a homomorphism $\eta \colon \SL_2 \to G_s$ 
so that, roughly speaking, the simple coroot $\ach$ of $\SL_2$ 
(viewed as a map $\Gm \to T_1 := \eta^{-1}(T)$) 
satisfies $\eta(\ach) = \dch$.  
We check that the diagram
\begin{equation} \label{mc.1}
\begin{CD}
\Z[T^*] @>\phi_2>> \sa(T^*)/(\sa^3(T^*)) \\
@V{\eta^*}VV @VV{\eta^*}V \\
\Z[T_1^*] @>{\phi_2}>> \sa(T_1^*)/(\sa^3(T_1^*))
\end{CD}
\end{equation}
commutes.  
Since $\omega_j(\dch) = d_j$, 
we have $\eta^*(\omega_j) = d_j \omega$ 
for $\omega$ the fundamental weight of $\SL_2$ dual to $\ach$.  
We find:
\[
\eta^* \phi_2( e^{\sum c_j \omega_j}) = 
\prod\nolimits_j (1 - d_j \omega)^{-c_j} = (1- \omega)^{-\sum c_j d_j},
\]
because the $d_j$ are all 0 or 1.  
As this is $\phi_2(e^{(\sum d_j c_j)\omega}) = \phi_2 \eta^*(e^{\sum c_j \omega_j})$, 
we have confirmed the commutativity of \eqref{mc.1}.

Put $\phi^2$ for the composition of $\phi_2$ 
with the projection onto the degree $2$ component $\sa^2$, 
so $q_1 = \phi^2(\chi)$. 
Then $q_1(\dch) = (\eta^* \phi^2(\chi))(\ach)$ obviously, 
which is $(\phi^2 \eta^*(\chi))(\dch)$ by commutativity of \eqref{mc.1}. 
We have $\eta^*(\chi) = \sum_j e^{\sum_i a_{ij} \omega}$ 
and by Lemma \ref{equinum}, 
the multiset of the $j$ integers $\sum_i a_{ij} d_i$ 
is symmetric under multiplication by $-1$, 
hence by Example \ref{A1.eg} we find:
\[
q_1(\dch) = 
\frac12 \left( \sum\nolimits_j \left( \sum\nolimits_i a_{ij} d_i \right)^2 \right) = 
\frac12 \sum\nolimits_j \qform{\la_j, \dch}^2.
\]
By \eqref{longroot} this equals $q(\dch) N(\chi) = q_2(\dch)$.

\smallskip
Returning to equation \eqref{mc.q}, this shows that 
the term $q_\ell(\ach_i)$ does not depend on $\ell$.  
Similarly, if $\ach_i$ and $\ach_j$ are not orthogonal coroots, 
then $\ach_i$ and $\ach_j$ are adjacent in the Dynkin diagram 
and $\ach_i + \ach_j$ is a coroot \cite[VI.1.6, Cor.~3b]{Bou:g4}.  
The preceding two paragraphs show that the value of 
\[
b_{q_\ell}(\ach_i, \ach_j) = q_\ell(\ach_i + \ach_j) - q_\ell(\ach_i) - q_\ell(\ach_j)
\]
does not depend on $\ell$.

It remains to consider $b_{q_\ell}(\ach_i, \ach_j)$ 
where $\ach_i$ and $\ach_j$ are orthogonal 
(relative to the polar form of $q$ -- 
it follows that they are orthogonal relative to $b_{q_2}$.  
We use $\ach_i$ and $\ach_j$ to define a homomorphism 
$\tau\colon \SL_2 \times \SL_2 \to G_s$ 
and -- as we did for $\SL_2$ above -- we fix a torus 
$T_2 = T_1 \times T_1 \subset \SL_2 \times \SL_2$ such that 
$\tau(T_2) = \im (\ach_i \times \ach_j) \subset T$.  
Arguing using a commutative diagram analogous to \eqref{mc.1}, 
it suffices to check that the simple roots of $\SL_2 \times \SL_2$ 
are orthogonal relative to $\tau q_1 = \phi^2 \tau(\chi)$, 
which follows from Example \ref{A1A1.eg} and Lemma \ref{equinum}.
\end{proof}

In view of \S\ref{index.def},
Prop.~\ref{maincomp} gives:

\begin{cor} \label{mc.cor}
$\phi_2(I') = \Z \cdot N(G_s) \cdot q$.
$\hfill\qed$
\end{cor}

\begin{ex} \label{A.eg}
Suppose $G$ has type $A_n$ for some $n \ge 2$.  We continue the notation of \ref{deg3}, and we compute:
\begin{align*}
\phi_3(We^{\omega_1}) &= \phi_3(e^{\omega_1} + e^{-\omega_n} + \sum_{i=2}^n e^{\omega_i - \omega_{i-1}}) \\
&= (n+1) + q + \sum_{j=1}^n \omega_i^3 - \sum_{i=2}^n \omega_{i-1} \omega_i^2.
\end{align*}
The element $\D := We^{\omega_1} - We^{\omega_n}$ is in $I'$, and by symmetry we see that $\phi_3(\D) = f_3$.
\end{ex}

%%%%%%%%%%%%%%%%%%%%%%%%%%%%%%%%%%%%%%%%%%%%%%%%%%%%%%%%%%%

\section{Torsion in the $\gamma$-filtration} \label{Bor.sec}

Let $\BB$ denote 
the variety of Borel subgroups of a strongly inner simple
linear algebraic group $G$ over $k$.
Recall that $G$ is  \emph{strongly inner} 
if the simply connected cover of $G$ is isomorphic to 
$G_s$ twisted by a cocycle $\xi \in H^1_{\text{\'et}}(k, G_s)$, 
where $G_s$ denotes the simply connected split group 
of the same Killing-Cartan type as $G$.
Observe that the variety $\BB$
is always defined over $k$ by \cite[Cor.~XXVI.3.6]{SGA3.3};
it is a twisted form
of the variety of Borel subgroups $\BB_s$ of $G_s$,
i.e., $\BB$ and $\BB_s$ become isomorphic over the algebraic
closure of $k$.

In the present section we determine and bound respectively
the torsion parts
of the second and the third quotients
of the $\gamma$-filtration on the variety $\BB$. 
The main result is the following.

\begin{thm}\label{Borprop}
Let $\BB$ be the variety of Borel subgroups
of a strongly inner simple linear algebraic group $G$ 
over a field $k$.
Then 
\begin{enumerate}
\renewcommand{\theenumi}{\roman{enumi}}
\item \label{Bor2}
$\Tors \gg23(\BB)$ 
is a cyclic group of order the Dynkin index $N(G)$ and generated by 
$\cc'(\theta)$ for $\theta$ as in Def.~\ref{specc}.

\item \label{versal} 
The subgroup $\tau^3(\BB) / \gamma^3(\BB)$ of $\Tors \gg23(\BB)$ is generated by $o(r(G))\,\cc'(\theta)$; it is cyclic of order
$N(G) / o(r(G))$.

\item \label{Bor3}
$2 \Tors \gg34(\BB)$ is a quotient of $(\Z/N(G))^{\oplus (\rank G)}$.
\end{enumerate}
\end{thm}

There is some new notation in the statement of the theorem.  The \emph{Rost invariant} $r$ is a map 
$H^1(k, G_s) \to H^3(k, \Q/\Z(2))$, 
and for our $G$, the element $r(\xi)$ depends only on $G$ 
and not on the choice of $\xi$ by \cite[Lemma 2.1]{GPe}; 
we write simply $r(G)$ for this element and $o(r(G))$ for its order in the abelian group $H^3(k, \Q/\Z(2))$.

Philippe Gille pointed out to us that 
pasting together two results in the literature 
gives a description of $\Tors \CH^2(X)$ 
for some $X$.  
\begin{prop} \label{Gille.prop}
Let $X$ be a projective homogeneous variety 
under $G$.  If $G$ is split by $k(X)$,
then
$\Tors \CH^2(X)$ is a cyclic group 
whose order is the same as the order of $r(G)$ in $H^3(k, \Q/\Z(2))$; 
in particular its order divides $N(G)$.
\end{prop}

\begin{proof}
We view $\xi$ as a principal homogeneous $G_s$-variety.
The kernel of the scalar extension map 
$H^3(k, \Q/\Z(2)) \to H^3(k(\xi), \Q/\Z(2))$ 
is the cyclic group generated by $r(G)$ by \cite[p.~129]{GMS}.  
For every extension $L/k$, $\xi$ has  an $L$-point 
if and only if $G$ is split, if and only if $X$ has an $L$-point.  
Therefore, this kernel is the same as the kernel of 
$H^3(k, \Q/\Z(2)) \to H^3(k(X), \Q/\Z(2))$.  
A theorem of Peyre-Merkurjev \cite{Peyre:deg3} shows that 
this kernel is isomorphic to $\Tors \CH^2(X)$.
\end{proof}

Obviously, one can take $X = \BB$ in the proposition.  
Furthermore,
the same proof shows that the proposition still holds 
if one replaces ``$G$ is strongly inner" 
and ``$G$ is split by $k(X)$" with ``$G$ has trivial Tits algebras" 
and ``$G$ becomes quasi-split over $k(X)$".

Also, statement \ref{Borprop}\eqref{Bor2} makes use of the following definition.
\begin{dfn} \label{specc}
Write the form $q$ from subsection \ref{q.def} 
(relative to the split group $G_s$) as 
$q = \sum_{i \le j} c_{ij} \omega_i \omega_j \in \sa(T^*)$.  
We call the element
\[
\theta := 
\sum_{i \le j} c_{ij} (1 - e^{-\omega_i})(1 - e^{-\omega_j}) \quad \in \Z[T^*]
\]
the \emph{special cycle}.  
Its image in $\Z[T^*]/\gamma^{m+1}$ is $\psi_m(q)$ for all $m \ge 2$.
\end{dfn}

\begin{proof}[Proof of Theorem \ref{Borprop}]
By the result of Panin \cite[Thm.~2.2.(2)]{Pa94}
since $G$ is strongly inner, 
the restriction map 
\begin{equation} \label{res.map}
\res\colon K_0(\BB) \ra K_0(\BB \times \kalg) \simeq K_0(\BB_s \times \kalg) \simeq K_0(\BB_s)
\end{equation}
is an isomorphism, where $\kalg$ denotes an algebraic closure of $k$.
Since the $\gamma$-filtration is defined in terms
of Chern classes and the latter commute with restrictions, 
it induces an isomorphism
between the $\gamma$-quotients, i.e., 
$$
\res\colon\gg{i}{i+1}(\BB)\xrightarrow{\simeq} \gg{i}{i+1}(\BB_s).
$$
Therefore, we may reduce to the split case $G=G_s$.  
Let $T$, $T^*$, etc., be as in subsection \ref{q.def}.

There is a commutative diagram
\begin{equation}\label{commdiag}
\xymatrix@C1.25in{
\gamma^{m/m+1}(\BB_s)\ar[r]^{c_m}  & \CH^m(\BB_s) \\
\gamma^m/\gamma^{m+1} 
\ar@{->>}[u]^{\cc'} \ar[r]^-{(-1)^{m-1}(m-1)!\cdot \phi_m} & 
\sa^m(T^*)/(\sa^{m+1}(T^*))\ar[u]_{\cc}
}
\end{equation}

First take $m = 2$ and suppose that $x \in \gg23$ 
maps to a torsion element in $\gg23(\BB_s)$.  
As $\CH^2(\BB_s)$ has zero torsion, 
the commutativity of \eqref{commdiag} shows that 
$\phi_2(x)$ is in the kernel $I$ of $\cc$.  
Writing $x = \sum_{i, j} a_{ij} (1 - e^{\omega_i})(1-e^{\omega_j}) \mod \gamma^3$, 
we have $\phi_2(x) = \sum a_{ij} \omega_i \omega_j$ of degree 2 in $I$, 
hence $\phi_2(x) = aq$ for some $a \in \Z$.  
Then modulo $\gamma^3$, we have 
$x \equiv \psi_2\phi_2(x) \equiv a\theta$, so 
$\Tors \gg23(\BB_s)$ is a cyclic group 
generated by the class of the special cycle $\theta$ modulo $\gamma^3 + I'$.

By Corollary~\ref{mc.cor} there exists $\chi\in I'$ 
such that $\phi_2(\chi)= N(G_s)\cdot q$.
Applying $\psi_2$ we obtain that
$$
0\equiv \chi \equiv N(G_s)\cdot \theta\mod \gamma^3+I',
$$
hence, the order of $\theta$ modulo $\gamma^3 + I'$ 
divides the Dynkin index $N(G_s)$.
This shows that $\Tors \gg23(\BB)$ is a cyclic group of order dividing $N(G)$ with generator $\cc'(\theta)$.

\smallskip
Let $\xi' \in H^1(k', G_s)$ be a versal $G_s$-torsor for some extension $k'$ of $k$, and write $\BB'$ for the Borel variety (over $k'$) of the group $G_s$ twisted by $\xi'$.  The element $r(\xi')$ has order $N(G_s)$ in $H^3(k', \Q/\Z(2))$ by  \cite[pp.~31, 133]{GMS}.  But $\Tors \gg23(\BB')$ is cyclic of order dividing $N(G_s)$, hence Prop.~\ref{Gille.prop} and the exactness of \eqref{deg2.seq} give that $\Tors \gg23(\BB')$ also has order $N(G_s)$. 
Now take $K$ to be an algebraically closed field containing $k'$.  The restriction maps for $k \ra K$ and $k' \ra K$ give isomorphisms $\Tors \gg23(\BB_{/k}) \simeq \Tors \gg23((\BB_s)_{/K}) \simeq \Tors \gg23(\BB'_{/k'})$, which is itself $\Z/N(G)$, completing the proof of \eqref{Bor2}.  Claim \eqref{versal} follows from the exactness of \eqref{deg2.seq}.

\smallskip
Now take $m = 3$ and suppose that $x \in \gg34$ 
maps to a torsion element in $\gg34(\BB_s)$.  
As $\CH^3(\BB_s)$ has zero torsion, 
diagram \eqref{commdiag} shows that $2\phi_3(x)$ 
is in the kernel $I$ of $\cc$.  
As in the $m = 2$ case, $2\phi_3(x)$ has degree 3.

Suppose $G_s$ is not of type $A_n$ for $n \ge 2$.  Then by \S\ref{deg3},
$2\phi_3(x)=q\cdot f$, 
where $f=\sum_{i=1}^n a_i\omega_i$.
Applying $\psi_3$
we obtain that $2x=\theta\cdot f'$, 
where $f'=\sum_{i=1}^n a_i(1-e^{-\omega_i})$.
In other words, the torsion part of $2\gg34(\BB_s)$ is generated
by the elements 
$\cc'(\theta\cdot (1-e^{-\omega_i}))$ for $i=1, \ldots, n$.

By Corollary~\ref{mc.cor} there exists $\chi\in I'$ such that 
$\phi_3(\chi\cdot (1-e^{-\omega_i}))\equiv 
N(G_s)\cdot q \cdot \omega_i \mod (\sa^4(T^*))$.
Applying $\psi_3$ we obtain that
$$
0\equiv \chi\cdot (1-e^{-\omega_i}) \equiv 
N(G_s)\cdot \theta\cdot (1-e^{-\omega_i})
\mod \gamma^4+I',
$$
hence, the torsion part of $2\gg34(\BB_s)$ 
is a product of $n$ cyclic groups of orders dividing the $N(G_s)$.

\smallskip
It remains to consider the case $m = 3$ and $G_s$ of type $A_n$ for $n \ge 2$; the claim is that $2x$ is in $I'$.  As in the preceding paragraph, \S\ref{deg3} and Example \ref{A.eg} show that $2\phi_3(x) = q\cdot f + b\phi_3(\D)$ for some $b \in \Z$.  Applying $\psi_3$, we find $2x = \theta \cdot f' + \D \cdot b$.  As $\theta$ and $\D$ belong to $I'$, the proof of \eqref{Bor3} is complete.
\end{proof}

\begin{rem} Observe that in the proof we used the structure of the
ideal of invariants $I$ in degrees 2 and 3. In principle it is possible
to extend the proof to higher degrees, but then one must extend the arguments of \S\ref{deg3} and Example \ref{A.eg}.
\end{rem}

The next lemma can be used to extend the bounds 
obtained in Th.~\ref{Borprop} 
to the case of a semisimple group.  

\begin{lem}\label{prodb} 
Let $G_1, \ldots, G_m$ be simple and 
strongly inner groups and write $\BB_j$ 
for the Borel variety of $G_j$.  
The Borel variety for $\prod G_j$ is $\prod \BB_j$, and we have:
\[
\Tors \gg23(\prod \BB_j) \simeq \bigoplus \Tors\gg23(\BB_{j}),
\] and
\[
\Tors\gg34(\prod \BB_j)\simeq 
\bigoplus_{j=1}^m\big(\Tors\gg34(\BB_{j})\oplus
\Tors\gg23(\BB_{j})\big).
\]
\end{lem}

\begin{proof} Apply the K\"unneth decomposition and the fact that 
$\gg{i}{i+1}(\BB_j)$ has no torsion for $i=0$ and $1$.
\end{proof}

%%%%%%%%%%%%%%%%%%%%%%%%%%%%%%%%%%%%%%%%%%%%%%%%%%%%%%%%%%%

\section{Examples of torsion in $\CH^2$} \label{CH2.sec}

We now make a few remarks regarding torsion in $\CH^2$.    We maintain the notation of the previous section.

Using Prop.~\ref{Gille.prop}, one can view the
the assignment $G\mapsto \Tors \CH^2(\BB)$
as a replacement 
for the Rost invariant of $G$. 
Furthermore, the group
$\Tors\CH^2(\BB)$ is generated by the image
of the special cycle $\theta \in \Z[T^*]$ 
which is defined in purely combinatorial terms.

The next two examples show that the image of 
$\theta$ generates the torsion in $\CH^2$ for certain
(generalized) Rost motives. 

\begin{ex}
Assume that the motive of $\BB$ splits as 
a direct sum of twisted copies of 
the Rost motive $\R_2$
corresponding to a $3$-fold Pfister form. 
According to \cite[\S7]{PSZ:J} this can happen for $G$
of type $B_n$, $D_n$, $G_2$, $F_4$, or $E_6$.
Then we have 
$$(\Tors \CH^2(\BB))\ot\Z/2\Z\simeq \Ch^2(\R_2),$$
where $\Ch$ denotes the Chow group with $\Z/2\Z$-coefficients and
the image of the special cycle $\theta$ 
generates $(\Tors\CH^2(\BB))\ot\Z/2\Z$ and, hence, 
the group $\Ch(\R_2)$. 
\end{ex}

\begin{ex}\label{mottor}
Suppose $G$ is an anisotropic group of type $F_4$ over $k$ 
that is split by a cubic extension of $k$.  
(Such a group exists if and only if $H^3(k, \Z/3\Z(2))$ is not zero.)
By \cite[Rem.~4.5]{NSZ} and the main theorem of 
\cite{PSZ:J} the motive of the Borel variety $\BB$
splits as a direct sum of generalized Rost motives 
$\R_3$ corresponding
to the Rost-Serre invariant $g_3$ 
(the mod-3 portion of the Rost invariant) of $G$.
Therefore, we have 
$$(\Tors \CH^2(\BB))\ot\Z/3\Z \simeq \Ch^2(\R_3),$$
where $\Ch$ denotes the Chow group with $\Z/3\Z$-coefficients.

By the recent results 
of Merkurjev-Suslin \cite{MS10} and Yagita \cite[Thm.~10.5, Cor.~10.8]{Yagita:torsion}
we have
$\Ch^2(\R_3)\simeq \Z/3\Z$.
On the other hand, the image of the special cycle $\theta$
generates $(\Tors\CH^2(\BB))\ot \Z/3\Z$ and, 
hence, the group $\Ch^2(\R_3)$. 
\end{ex}

%%%%%%%%%%%%%%%%%%%%%%%%%%%%%%%%%%%%%%%%%

\section{Torsion in $\CH^3$} \label{CH3.sec}

Let $X$ be a projective homogeneous $G$-variety 
such that $G$ is split over $k(X)$.  (``$X$ is generically split.")
Thanks to Proposition \ref{Gille.prop}, 
we may view $\Tors \CH^2(X)$ as known, so we now investigate $\Tors \CH^3(X)$.
We retain the meaning of $G$ and $\BB$ from the previous section. 
Let $n$ denote the rank of $G$ and let $r$ denote the rank
of the Picard group of $X$ over an algebraic closure of $k$.

We remark that the results in this section only use the fact that $\Tors \CH^2(X)$ is cyclic of order dividing $N(G)$, which follows from our Theorem \ref{Borprop}\eqref{Bor2} and Eq.~\eqref{deg2.seq}.  They do not need the finer result of Prop.~\ref{Gille.prop}, hence also do not need material from \cite{GMS} and \cite{Peyre:deg3}.

For an abelian group $A$ and a prime $p$, write $\Tors_p A$ for the subgroup of $A$ consisting of elements of order a power of $p$.

\begin{lem} \label{tau}
The restriction of the $m$-th Chern class gives a surjection
\[
\Tors \tau^{m/m+1}(X) \twoheadrightarrow (m-1)! \Tors \CH^m(X)
\]
and for each prime $p$ not dividing $(m-1)!$, $c_m$ is an isomorphism
\[
\Tors_p \tau^{m/m+1}(X) \xrightarrow{\simeq} \Tors_p \CH^m(X).
\]
\end{lem}

\begin{proof}
By Riemann-Roch (see subsection \ref{twofil}), the composition 
\[
\xymatrix{
\CH^m(X) \ar@{->>}[r]^(0.45)\pr & \tau^{m/m+1}(X) \ar[r]^(0.51){c_m} & \CH^m(X) }
\]
is multiplication by $(-1)^{m-1} (m-1)!$, hence $c_m(\tau^{m/m+1}(X))$ is $(m-1)! \CH^m(X)$.  For $x \in \Tors \CH^m(X)$, we have $(m-1)! \cdot x = c_m(\pr(x))$, where $\pr(x)$ is in $\Tors \tau^{m/m+1}(X)$.  This proves the first claim, from which the second claim follows immediately.
\end{proof}

\begin{prop} \label{corbor3} 
If $\tau^3(\BB) = \gamma^3(\BB)$, then 
$\Tors 4\cdot \CH^3(\BB)$ is a quotient of the direct sum  $(\Z/N(G)\Z)^{\oplus n}$.
In particular, 
the torsion part of $\CH^3(\BB)$ can consist
only of subgroups $\Z/2^s\Z$ for $s\le 4$, $\Z/3\Z$, or $\Z/5\Z$.
\end{prop}

Theorem \ref{Borprop}\eqref{versal} gives a way to check the hypothesis on the filtration.

\begin{proof}
By the hypothesis, 
the map
$\gg34(\BB) \to \tau^{3/4}(\BB)$ is surjective.  Now combine Lemma \ref{tau} and Theorem \ref{Borprop}\eqref{Bor3}.
\end{proof}

As an alternative to making a hypothesis on the filtrations, 
we may control the torsion on $\CH^3(X)$ based on information 
about the torsion in $\CH^2(X)$ and the motivic decomposition of $X$,
as we now illustrate.

Fix a prime $p$.  In the category of Chow motives with $\Zp$-coefficients, 
the motive of $X$ is a direct sum of shifts of an indecomposable motive $\R$, 
see \cite[Th.~5.17]{PSZ:J}, 
where $\R$ depends on $G$ but not the choice of $X$ (ibid., Th.~3.7).  
We write $\Ch^m(\R)$ for the $m$-th Chow group of $\R$ with $\Zp$ coefficients.

\begin{lem}\label{gensplitr} 
We have: 
\begin{enumerate}
\renewcommand{\theenumi}{\roman{enumi}}
\item \label{gensp2r} $(\Tors_p \CH^2(X)) \ot \Zp \simeq \Ch^2(\mathcal{R})$; 
\item \label{gensp3r} $(\Tors_p \CH^3(X)) \ot \Zp \simeq 
(\Ch^2(\mathcal{R}))^{\oplus r}\oplus \Ch^3(\mathcal{R})$.
\item \label{gensp3} $\Tors \CH^3(\BB) \simeq (\Tors \CH^2(X))^{\oplus (n - r)} \oplus \Tors \CH^3(X)$.
\end{enumerate}
\end{lem}

\begin{proof} 
The expression of the motive of $X$ from \cite{PSZ:J} gives:
\[
(\Tors_p \CH^m(X)) \otimes\Zp \simeq \Chb^m(\R)\oplus (\Chb^{m-1}(R))^{\oplus r}\oplus
(\Chb^{m-2}(\mathcal{R}))^{\oplus \ldots}\oplus\ldots
\]
where $\Chb^m(\R)$ denotes the kernel of the restriction
$\Ch^m(\R)\to \Ch^m(\R\times_k \bar k)$ to the algebraic closure $\bar k$.
By the formula for the generating function \cite[Thm.~5.13(3)]{PSZ:J}
and table 4.13 in ibid., we have
$\Chb^0(\R)=\Chb^1(\R)=0$ and
$\Chb^i(\R)=\Ch^i(\mathcal{R})$ for $i=2,3$.
This implies claims \eqref{gensp2r} and \eqref{gensp3r}. 

Claim \eqref{gensp3} is proved similarly, 
but using the integral motivic decomposition 
from \cite[Th.~3.7]{PSZ:J} with $Y = \BB$.
\end{proof}

\begin{prop}\label{genspt}
Fix an odd prime $p$.
If $\Tors_p \CH^2(X) \ne 0$, then 
\begin{enumerate}
\item $p = 3$ or $5$;
\item \label{genspt.R} $\Ch^2(\R) \simeq \Zp$ and $\Ch^3(\R)=0$.
\item \label{genspt.CH} $\Tors_p \CH^2(X) \simeq \Zp$ and $\Tors_p \CH^3(X) \simeq (\Zp)^{\oplus r}$.
\end{enumerate}
\end{prop}

\begin{proof}
By Prop.~\ref{Gille.prop} (or \cite[Th.~3.7]{PSZ:J}), 
$\CH^2(X)$ and $\CH^2(\BB)$ have the same $p$-torsion.
As $\Tors \CH^2(\BB)$ has order dividing $N(G)$ 
by Theorem \ref{Borprop}\eqref{Bor2},
the list of Dynkin indexes in \ref{index.def} 
gives that $p = 3$ or 5 and $\Tors_p \CH^2(\BB) \simeq \Zp$.  
Combining this with Lemma \ref{gensplitr}\eqref{gensp2r}, 
it only remains to prove the claims about $\Ch^3(\R)$ and $\CH^3(X)$.

Tensoring sequence \eqref{deg2.seq} with $\Zp$, we find that
\[
\gamma^3(\BB)\otimes\Zp=
\tau^3(\BB)\otimes\Zp.
\]
Combining Lemma \ref{tau} and Theorem \ref{Borprop}\eqref{Bor3} 
gives that $\Tors_p \CH^3(\BB)$ is a product of at most $n$ copies of $\Zp$.
By Lemma~\ref{gensplitr}\eqref{gensp3r} applied to $X=\BB$
we obtain
$$
(\Tors_p \CH^3(\BB)) \ot \Zp \simeq (\Zp)^{\oplus n}\oplus \Ch^3(\mathcal{R}).
$$
Since the right hand side already contains $n$ copies of $\Zp$,
$\Tors_p \CH^3(\BB)=(\Zp)^{\oplus n}$ and $\Ch^3(\R)$ is zero.  
The second part of \eqref{genspt.CH} now follows by Lemma \ref{gensplitr}\eqref{gensp3}.
\end{proof}

%%%%%%%%%%%%%%%%%%%%%%%%%%%%%%%%%%%%

\section{Cohomological invariants and the Tits algebras}\label{sect:cohinv}

So far, we have studied the case where $G$ is strongly inner and we constructed the special cocycle $\cc'(\theta)$ in $K_0(\BB)$, cf.~Example \ref{SC.case} below.  We now relax our hypothesis on $G$ and ask if $\cc'(\theta)$ is still defined over $k$.

In the present section $G_s$ denotes 
an adjoint split simple linear algebraic group
 over a field $k$.  As it is adjoint, the character group $T^*$ of a split maximal torus of $G_s$ is naturally identified with the root lattice $\La_r$.
 
We fix a pinning for $G_s$, which includes a set of simple roots $\D = \{ \alpha_1, \ldots, \alpha_n\}$ in $\La_r$.  Write $\omega_i$ for the fundamental weight corresponding to $\alpha_i$ and $s_i$ for the reflection of the weight lattice $\La$ in the hyperplane orthogonal to $\alpha$.

\subsection*{The Steinberg basis.}
For each element $w$ of the Weyl group $W$ (of $T$) we define
\[
\rho_w:=\sum_{\{i\in 1\ldots n \mid w^{-1}(\alpha_i)<0\}} w^{-1}(\omega_i) \quad \in \La.
\]
Let $\Z[\Lambda]^W$ denote the subring of $W$-invariant elements.
Then the integral group ring $\Z[\Lambda]$ is a free $\Z[\Lambda]^W$-module
with the basis $\{e^{\rho_w} \mid {w\in W} \}$ by \cite[Th.~2.2]{St:Pittie}.

Let $\BB_s$ denote the variety of Borel subgroups of $G_s$.
Consider the characteristic map for the simply connected cover of $G_s$
$$
\cc' \colon \Z[\Lambda] \twoheadrightarrow K_0(\BB_s).
$$
Since the kernel of the surjection $\cc'$ is generated
by elements $x\in \Z[\Lambda]^W$ in the kernel of the augmentation map,
there is an isomorphism
\[
\Z[\Lambda]\otimes_{\Z[\Lambda]^W}\Z \simeq \Z[\Lambda]/\ker(\cc')\simeq 
K_0(\BB_s).
\]
The elements
$$
\{g_w:=\cc'(e^{\rho_w})=[\LL(\rho_w)] \mid {w\in W} \}
$$ 
form a free $\Z$-basis
of $K_0(\BB_s)$ called the {\em Steinberg basis}.

Observe that the quotient group $\Lambda/\Lambda_r$ 
coincides with the group of characters of the center of the 
simply connected cover of $G_s$.
Consider the surjective ring homomorphism induced by the 
restriction 
$\Z[\Lambda]\to \Z[\Lambda/\Lambda_r]$.
Since $W$ acts trivially on $\Lambda/\Lambda_r$, 
we obtain that
\[
\bar\rho_w=\sum_{\{i\in 1\ldots n \mid w^{-1}(\alpha_i)<0\}} \bar\omega_i,
\]
where $\bar \rho$ means the restriction to $\Lambda/\La_r$.

\begin{ex}\label{exStbas} 
(a) For a simple reflection $s_j$ we have
$$
\rho_{s_j}=\sum_{\{i\in 1\ldots n \mid s_j(\alpha_i)<0\}} s_j(\omega_i)=
s_j(\omega_j)=
\omega_j-\alpha_j.
$$

\begin{enumerate}
\setcounter{enumi}{1}
\renewcommand{\theenumi}{\alph{enumi}}
\item More generally, let $w := s_{i_1}s_{i_2}\ldots s_{i_m}$
be a product of $m$ distinct simple reflections such that
the simple roots $\alpha_{i_j}, \alpha_{i_\ell}$ are orthogonal for all $j \ne \ell$.
Then
\[
\rho_{s_{i_1}s_{i_2}\ldots s_{i_m}} =\rho_{s_{i_1}}+\rho_{s_{i_2}}+ \ldots +\rho_{s_{i_m}}.
\]
because $w^{-1}(\alpha_i)$ is negative if and only if $i = i_j$ for some $j$.

\item
For a product of two simple reflections $s_is_j$
such that $c_{ij}=\alpha_i^\vee(\alpha_j)<0$
we obtain
$$
\rho_{s_is_j} =
\rho_{s_i} + c_{ij}\alpha_j. 
$$
\end{enumerate}
\end{ex}

\subsection*{The Tits algebras and the base change}\label{basechSt}

Let $G$ be a twisted form of $G_s$, i.e.
$G$ is obtained by twisting  $G_s$ 
by a cocycle $\xi \in Z^1(k,\Aut(G_s))$.  More specifically, our choice of pinning for $G_s$ defines a section $s$ of the quotient map $\pi \!: \Aut(G_s) \ra \Aut(\D)$.  Twisting $G_s$ by $\xi' := s\pi(\xi)$ gives a quasi-split group $G_q$ and we pick $\xi'' \in Z^1(k, G_q)$ that maps via twisting to $\xi$---i.e., we pick $\xi''$ so that $G$ is isomorphic to ${_{\xi''}}G_q$.  

Let $\BB={}_\xi \BB_s$ be the variety of Borel subgroups of $G$.
Let $\Gamma$ denote the absolute Galois group of $k$; it acts on the weight lattice $\La$ via the cocycle $\xi'$.

Following \cite{Ti:R} (see also \cite[\S3.1, \S11.7]{Pa94} and \cite[\S2]{MePaWa}) we associate
with each $\chi \in \La/\La_r$ the field of definition $k_\chi$ of $\chi$ 
and the central simple algebra $A_{\chi,\xi}$
over $k_\chi$ called the Tits algebra. 
Here $k_\chi$ is a fixed subfield for the stabilizer 
$$
\Gamma_\chi=\{\tau\in \Gamma\mid \tau (\chi)=\chi\}.
$$
There is a group homomorphism 
$$
\beta \!: (\La/\La_r)^{\Gamma_\chi} \to \Br(k_\chi)\text{  with } 
\beta(\chi')=[A_{\chi',\xi}].
$$
By \cite[Thm.~2.1]{Peyre:deg3} there is an isomorphism
$$
\Tors\CH^2(\BB)\simeq \frac{\ker \left( H^3(k,\Q/\Z(2))\to H^3(k(\BB),\Q/\Z(2)) \right)}
{\bigoplus_{\chi\in \La/\La_r}  N_{k_\chi/k}\big(k_\chi^* \cup \beta(\chi)\big) },
$$
where the numerator is the kernel of the restriction map to the field of fractions $k(\BB)$ of $\BB$
and $N_{k_\chi/k}$ is the norm map. Let $H^3_{\beta}(k,\Q/\Z(2))$ denote the cohomology quotient
$$
H^3_\beta (k,\Q/\Z(2))=H^3(k,\Q/\Z(2))/\bigoplus_{\chi\in \La/\La_r}  N_{k_\chi/k}\big(k_\chi^* \cup \beta(\chi)\big)
$$
so that $\Tors \CH^2(\BB) \subseteq H^3_\beta(k,\Q/\Z(2))$.

Let $l/k$ be a field extension. Since the Chern classes commute with restrictions, there is the induced map
$$
\res_{l/k}\colon \gamma^{i/i+1}(\BB)\to \gamma^{i/i+1}(\BB_l),
$$
where $\BB_l=\BB\times_k l$,
with the image generated by the products
$$
\langle c_{n_1}^{K_0}(x_1)\cdots  c_{n_m}^{K_0}(x_m) 
\mid n_1+\cdots +n_m= i,\; x_1,\ldots,x_m\in \res_{l/k}\big(K_0(\BB)\big)\rangle
$$
and there is a commutative diagram
\begin{equation}\label{cohtors}
\xymatrix{
\Tors \gamma^{2/3}(\BB) \ar@{>>}[r]^-{c_2}\ar[d]_{\res_{l/k}}   & \Tors \CH^2(\BB) \ar[d]^{\res_{l/k}} \subseteq H^3_\beta(k,\Q/\Z(2)) \\
\Tors \gamma^{2/3}(\BB_l) \ar@{>>}[r]^-{c_2}& \Tors \CH^2(\BB_l) \subseteq H^3_\beta(l,\Q/\Z(2))
}
\end{equation}

Observe that the image $\res_{l/k}\big(K_0(\BB) \big)$ can be computed using \cite{Pa94}.
For instance, if $G$ is an inner group, i.e., $\xi' = 0$, then  $\Gamma$ acts trivially on 
$\La/\La_r$, i.e. $k_\chi=k$ for all $\chi$ and
by \cite[Thm.~4.2]{Pa94}
the image of the restriction map $K_0(\BB) \to K_0(\BB_s)$ from \eqref{res.map} coincides with the sublattice 
$$
\langle \ind(A_{\bar\rho_w,\xi}) \cdot g_w \mid w\in W \rangle.
$$

Using \eqref{cohtors} one can provide a non-trivial element
in $H^3_\beta(k,\Q/\Z(2))$ as follows:
\begin{itemize}
\item
Assume that we are given non-trivial elements over $l$, i.e. that there is a non-trivial
element $\theta \in \Tors \gamma^{2/3}(\BB_l)$ such that $c_2(\theta) \in H^3_\beta(l,\Q/\Z(2))$ is non-trivial.
\item
Assume that we know that $\theta$ is defined over $k$, i.e. that
$\theta=\res_{l/k}(\theta')$ for some $\theta'\in \Tors \gamma^{2/3}(\BB)$.
\end{itemize}
Then the image $c_2(\theta')$ provides a non-trivial element in $H^3_\beta(k,\Q/\Z(2))$.

\begin{ex}[strongly inner case]  \label{SC.case}
If $G$ is strongly inner---i.e., if $G$ is inner and $\beta$ is the trivial homomorphism---then 
for any field extension $l/k$ the left vertical arrow in \eqref{cohtors} is an isomorphism, hence,
identifying $\Tors\gamma^{2/3}(\BB)$ with the cyclic group generated by the special cycle $\theta$. 
As in Prop.~\ref{Gille.prop} and its proof
$\Tors \CH^2(\BB)$ coincides with the usual unramified cohomology generated by the Rost invariant 
$r(G)$ of $G$ and
$\langle c_2(\theta)\rangle=\langle r(G)\rangle$ in $H^3(k,\Q/\Z(2))$. 
\end{ex}

\begin{lem} \label{sumrat} Assume that $G$ is inner. 
\begin{itemize}
\item[(a)]
If a weight $\omega$ is such that $\beta(\omega) = 0$, 
then $[\LL(\omega)]$ is in the image of $$\res \!: K_0(\BB) \ra K_0(\BB_s).$$
In particular, it holds for the classes $[\LL(\alpha_i)]$ 
of simple roots $\alpha_i$.
\end{itemize}

Under the notation 
of Example \ref{exStbas}(b) we have  
\begin{itemize}
\item[(b)]
$
\sum_j c_1^{K_0}([\LL(\omega_{i_j})]) -c_1^{K_0}([\LL(\alpha_{i_j})])  
\equiv c_1^{K_0}\left( \prod_j g_{s_{i_j}} \right) \equiv c_1^{K_0}(g_w) \mod \gamma^3(\BB_s);
$
\item[(c)]
If $\beta(\sum_j \omega_{i_j}) = 0$, then $\sum_j c_1^{K_0}([\LL(\omega_{i_j})])$ is in the image of $$\res \!: \gg12(\BB) \ra \gg12(\BB_s).
$$
\end{itemize}
\end{lem}

\begin{proof}
(a) follows by  \cite[Cor.~3.1]{GiZ}. 
(b) follows by
the formula for the first Chern class (in $K_0$) 
of the tensor product of line bundles.
According to (a)
each $c_1^{K_0}([\LL(\alpha_{i_j})])$ is in the image of the restriction map
which implies (c).
\end{proof}

\begin{prop}[quaternionic inner case] \label{quaternionic}
Assume that $G$ is inner.
If every Tits algebra of $_\xi G_s$ has index 1 or 2, then the special cycle $\theta$ is in the image of the restriction map
\[
\res\colon \gamma^{2/3}(\BB)\to \gamma^{2/3}(\BB_s).
\]
In other words, if $l/k$ is an extension that kills $\im \beta$, then the image of $c_2$ over $l$
coincides with the subgroup generated by the respective Rost invariant, i.e. we have
$$
\im (c_2)_l=\langle r(G_l)\rangle \subseteq H^3(l,\Q/\Z(2)).
$$
\end{prop}

\begin{proof}[Proof of Prop.~\ref{quaternionic}]
We may assume that $N(G)$ is not $1$ (otherwise $\theta$ maps to zero in $\gg23(\BB_s)$ by Th.~\ref{Borprop}) and $\La/\La_r$ has even order (otherwise Example \ref{SC.case} applies), i.e., we may assume that $G$ has type $B$, $C$, $D$, or $E_7$.

We first make a general observation.  Mod $\gamma^3(\BB_s)$, we have:
\begin{align*}
c_1^{K_0}([\LL(\omega_i)])^2 &\equiv (c_1^{K_0}(g_{s_i}) + c_1^{K_0}([\LL(\alpha_i)]))^2 \\
&\equiv c_1^{K_0}(g_{s_i})^2 + 2c_1^{K_0}(g_{s_i}) c_1^{K_0}([\LL(\alpha_i)]) + c_1^{K_0}([\LL(\alpha_i)])^2.
\end{align*}
The Whitney Sum Formula gives that $c_2^{K_0}(2g_{s_i}) = c_1^{K_0}(g_{s_i})^2$ and $c_1^{K_0}(2g_{s_i}) \equiv 2c_1^{K_0}(g_{s_i}) \mod{\gamma^2(\BB_s)}$.  Our hypothesis on the Tits algebras gives that $2g_{s_i}$ is in the image of $K_0(\BB) \ra K_0(\BB_s)$, and it follows that $c_1^{K_0}([\LL(\omega_i)])^2$ is \emph{rational} -- i.e., is in the image of $\gg23(\BB) \ra \gg23(\BB_s)$ -- for all $i$.

\smallskip

\fcase{Type $E_7$} Suppose that $G$ has type $E_7$.  Then 
\[
q = \left( \sum_{i=1}^7 \omega_i^2 \right) -  \omega_1 \omega_3 - \omega_3 \omega_4 - \omega_4 \omega_2 - \omega_4 \omega_5 - \omega_5 \omega_6 - \omega_6 \omega_7
\]
where we have numbered the roots following \cite{Bou:g4}.  Each $\omega_i^2$ contributes a term of the form $c_1^{K_0}([\LL(\omega_i)])^2$ to the image of $\cc'(\theta)$ in $\gg23(\BB_s)$, and such a term is rational by the preceding paragraph.  The weights $\omega_1, \omega_3, \omega_4, \omega_6$ belong to the root lattice and so the term $\omega_1 \omega_3$ contributes a rational term $c_1^{K_0}([\LL(\omega_1)]) c_1^{K_0}([\LL(\omega_3)])$ to $\cc'(\theta)$, and similarly for the term $\omega_3 \omega_4$.  Next we observe that $\omega_4 \omega_2 + \omega_4 \omega_5$ contributes
\[
c_1^{K_0}([\LL(\omega_4)]) \left( c_1^{K_0}([\LL(\omega_2)]) + c_1^{K_0}([\LL(\omega_5)]) \right)
\]
to $\cc'(\theta)$.  But $\omega_4$ and $\omega_2 + \omega_5$ both lie in the root lattice, so both terms in the product are rational by Lemma~\ref{sumrat}.  The same argument handles $\omega_5 \omega_6 + \omega_6 \omega_7$, and we are done with the $E_7$ case.

\smallskip

\fcase{Type $D$}  Suppose that $G$ has type $D_n$.  Then 
\[
q = \sum_{i=1}^n \omega_i^2 -  \sum_{i=1}^{n-2} \omega_i \omega_{i+1} - \omega_{n-2} \omega_n.
\]
The terms $\omega_i^2$ are treated as in the $E_7$ case.  For the terms in the second sum, we collect around terms with 
even subscripts: for even $i < n - 2$, consider $\omega_i (\omega_{i-1} + \omega_{i+1})$.  As $\omega_i$ and $\omega_{i-1} + \omega_{i+1}$ belong to the root lattice, we see as in the $E_7$ case that they contribute rational terms to $\cc'(\theta)$.

Suppose now that $n$ is even.  Then we have not accounted for $\omega_{n-2} (\omega_{n-3} + \omega_{n-1} + \omega_n)$ from $q$.  As both terms in the product belong to the root lattice, we are finished as in the $E_7$ case.

If $n$ is odd, then we have not accounted for $\omega_{n-2} (\omega_{n-1} + \omega_n)$ in $q$.  Here $\La/\La_r$ is isomorphic to $\Z/4$ and $\omega_{n-2}, \omega_{n-1}, \omega_n$ map to $2, \pm 1, \pm 3$ respectively.  In particular, $\beta(\omega_{n-2}) = 2\beta(\omega_n)$, which is zero by our hypothesis on the Tits algebras, so $[\LL(\omega_{n-2})]$ is in the image of $\res \!: K_0(\BB) \ra K_0(\BB_s)$.  Similarly, $\beta(\omega_{n-1} + \omega_n) = \beta(\omega_{n-1}) + \beta(\omega_n) = 0$, and as in the $E_7$ case, we see that $\cc'(\theta)$ is rational.

\smallskip

\fcase{Type $B$ or $C$} 
If $G$ has type $B_n$ or $C_n$, $\La/\La_r$ equals $\Z/2$.  In either case, 
\[
q = \sum_{i=1}^n c_{ii} \omega_i^2 - \sum_{i=1}^{n-1}2 \omega_i \omega_{i+1}.
\]
where the $c_{ii}$ are 1 or 2.  

For type $C_n$, the map $\La \ra \La/\La_r$ sends $\omega_i$ to the class of $i$.  Previous arguments easily handle the $n$ odd case.  If $n$ is even, previous arguments leave us to consider the term $2\omega_{n-1} \omega_n$.  But 
\begin{align*}
2c_1^{K_0}([\LL(\omega_{n-1})]) &\equiv 2(c_1^{K_0}(g_{s_{n-1}}) + c_1^{K_0}([\LL(\alpha_{n-1})])) \mod{\gamma^2(\BB_s)} \\
&\equiv c_1^{K_0}(2g_{s_{n-1}}) + 2c_1^{K_0}([\LL(\alpha_{n-1})]),
\end{align*}
and again we find that $\cc'(\theta)$ is rational.

For type $B_n$, the map $\La \ra \La/\La_r$ sends $\omega_n$ to 1 and all other fundamental weights to zero.  Consequently, it suffices to consider the term $2 \omega_{n-1} \omega_n$ in $q$.  For this we can apply the argument in the preceding paragraph.
\end{proof}

%%%%%%%%%%%%%%%%%%%%%%%%%%%%%%%%%%%%
\section{Application to essential dimension}

We now apply results from the previous section to give a lower bound on the essential dimension $\ed(G)$ for some algebraic groups $G$.  We refer to Reichstein's 2010 ICM lecture \cite{Rei:ICM} for a definition and survey of this notion.  Roughly speaking, it gives the number of parameters required to specify a $G$-torsor.

\begin{prop} \label{proped}
Let $G$ be an absolutely almost simple algebraic group.  Then $\ed(G) \ge 3$ unless $G$ is isomorphic to $\Sp_{2n}$ for some $n \ge 2$ (in which case $\ed(G) = 0$) or
$G$ has type $A$.
\end{prop}

The lower bound of 3 is in many cases very weak, but it has the advantage of being uniform and having a proof that is almost as uniform.  The existence of the Rost invariant gives the same lower bound on $\ed(G)$ when $G$ is simply connected, so our proposition can be viewed as removing the hypothesis ``simply connected" from that result.

\begin{rem}
For groups of type $A$, the lower bound is more complicated and we do not know what the answer is in every case.  Of course $\ed(\SL_n) = 0$.  Some other known cases are: For $n$ divisible by the square of a prime, we know that $\ed(\PGL_n) \ge 4$ by \cite[Th.~16.1(b)]{Rei:HJ}; and for ``intermediate" groups of type $A$ in good characteristic, the essential dimension is at least 4 by \cite[Th.~8.13]{RY}.  On the other hand, A.A.~Albert conjectured that central simple algebras of prime degree are cyclic, which would imply that $\ed(\PGL_n) = 2$ for square-free $n$.  However, this is only currently known for $n = 2$, 3, and 6.
\end{rem}

\begin{proof}[Proof of Prop.~\ref{proped}]
\fcase{Main case} Suppose first that $G$ is not of type $A$, $C$, nor $E_6$.  As essential dimension only goes down with field extensions, we may assume that $k$ is algebraically closed and bound $\ed(G_s)$ where $G_s$ is a split simple group not of type $A$ or $C$.  Put $\Gt_s$ for the simply connected cover of $G_s$.  Fix a versal $\Gt_s$-torsor $\xit \in H^1(L, \Gt_s)$ for some extension $L/k$.  Let $K$ be a field between $k$ and $L$ of minimal transcendence degree such that there is a $\xi \in H^1(K, G_s)$ whose image in $H^1(L, G_s)$ is the same as the image of $\xit$.

For sake of contradiction, suppose that $K$ has transcendence degree at most 2 over $k$.  By the hypothesis on the type of $G$, the Tits algebras of $_\xi G$ have exponent a power of 2 and so are actually of index 1 or 2 over $K$ by De Jong, see \cite{dJ} or \cite[Th.~4.2.2.3]{Lieblich:PI}.  By Proposition \ref{quaternionic}, there is a class $\psi \in \gamma^2(\BB)$ whose image under restriction to $L$ is $\cc'(\theta)$.  Now $\Tors \CH^2(\BB_K)$ is zero by Prop.~\ref{Gille.prop} because $H^3(K, \Q/\Z(2))$ is zero, and it follows that $\cc'(\theta)$ is zero in $\Tors \CH^2(\BB_L)$.  But $\cc'(\theta)$ has order $N(G_s) \ne 1$, a contradiction.

\smallskip
\fcase{Type $C$} If $G$ of type $C_n$ ($n \ge 2$) is simply connected, by hypothesis it is not $\Sp_{2n}$, so the Rost invariant is not zero on a versal $G$-torsor and the claim follows.  So suppose $G$ is adjoint.  (We give an argument that is characteristic-free; if $\car k \ne 2$, then $\ed(G) \ge n + 1$ by \cite[(1.1)]{ChSe}.)  If $n$ is odd then we can construct a nonzero normalized cohomological invariant of $\PSp_{2n}$ of degree 4 as in  \cite[Th.~4.1]{MacD:CIodd} and that case is settled.

It remains to show that $\ed(\PSp_{2n}) \ge 3$ when $k$ is algebraically closed and $n$ is even.  Let $k'$ be an extension of $k$ that has a quaternion division algebra $D$ and fix a class $\zeta \in H^1(k', \PSp_{2n})$ with image $[D] \in H^2(k', \mu_2)$ under the natural connecting homomorphism $\partial_{k'}$.  Fix a versal torsor $\xit \in H^1(L, {_\zeta \Sp_{2n}})$ for some extension $L/k'$.  Let $K$ be a field between $k$ and $L$ so that there is a class $\xi \in H^1(K, \PSp_{2n})$ with the same image as $\xit$ in $H^1(L, \PSp_{2n})$.  For sake of contradiction, suppose that $K$ has transcendence degree at most 2 over $k$.

By De Jong, $\partial_K(\xi)$ is the class of a quaternion algebra in $H^2(K, \mu_2)$.  So we could arrange from the beginning that $k' = K$ and $[D] = \partial_K(\xi)$.  We have a well-defined invariant
\[
\im \left[ H^1(*, {_\zeta \Sp_{2n}}) \ra H^1(*, {_\zeta \PSp_{2n}}) \right] \ra H^3(*, \Z/2\Z)
\]
defined on extensions of $K$ because the Rost invariant vanishes on $H^1$ of the center of every simply connected group of type $C_n$ with $n$ even \cite{GQ}.  On the other hand, the class of $\xi$ belongs to the domain of this map and the invariant does not vanish on it; this contradicts the hypothesis that $K$ has transcendence degree at most 2 over $k$.

The remaining case of type $E_6$ is known by arguments as in \cite[9.5, 9.7]{GiRei}.  Alternatively, one can repeat the ``main case" argument focusing on essential 2-dimension instead of essential dimension.
\end{proof}

Although we stated the proposition for absolutely almost simple groups, it quickly leads to the lower bound $\ed(G) \ge 3$ for more groups $G$.  We mention: (A) to obtain a lower bound on $\ed(G)$ for any linear group $G$, it suffices to give a lower bound on the essential dimension of the identity component of $G$ \cite[6.19]{BF}, so one needn't assume that $G$ is connected.  (B) If $k/k_0$ is a finite separable extension, one has $\ed(R_{k/k_0}(G)) \ge \ed(G)$, so the proposition immediately gives a similar (but slightly more complicated to state) result for groups that are simple but not absolutely simple. (C) If $G \simeq G_1 \times G_2$, then $\ed(G) \ge \max\{ \ed(G_1), \ed(G_2) \}$, and in this way we can weaken the hypothesis ``simple" to ``semisimple" at the cost of demanding that $G$ be adjoint or simply connected.

\smallskip
\noindent{\small{\textbf{Acknowledgments.} We thank Zinovy Reichstein and Mark MacDonald for their comments on an earlier version of this paper.  The first author's research was partially supported  by the NSF under grant DMS-0653502. The second author was supported by NSERC Discovery 385795-2010 and 
Accelerator Supplement 396100-2010 grants}}

\bibliographystyle{amsplain}

\bibliography{ch2}

\end{document}